\documentclass[12pt, twosided]{article}
\usepackage{latexsym,enumerate}
\usepackage{amsmath}
\usepackage{amsthm,amsopn}
\usepackage{amstext,amscd,amsfonts,amssymb,fontenc,bbm}
\usepackage[ansinew]{inputenc}
\usepackage{verbatim}
\usepackage{graphicx}
\usepackage{subfigure}

\newtheorem{innercustomthm}{Theorem}
\newenvironment{customthm}[1]
  {\renewcommand\theinnercustomthm{#1}\innercustomthm}
  {\endinnercustomthm}

\usepackage{pstricks,pst-node}
\usepackage{tikz}
\usepackage{pgf,tikz}
   \usepackage{epsfig}

\usepackage{lineno} 
\usepackage{textcomp}
\usepackage{mathtools} 

\newtheorem{theorem}{Theorem}[section]
\newtheorem{corollary}[theorem]{Corollary}

\newtheorem{proposition}[theorem]{Proposition}

\newtheorem{lemma}[theorem]{Lemma}

\theoremstyle{definition}

\DeclarePairedDelimiter\ceil{\lceil}{\rceil}
\DeclarePairedDelimiter\floor{\lfloor}{\rfloor}

\newcommand{\al}{\alpha}

\begin{document}

\title{Independence numbers of some double vertex graphs and pair graphs}

\author{Paloma Jim\'enez-Sep\'ulveda
\and Luis Manuel Rivera
}
\date{}

\maketitle

\begin{abstract}

The combinatorial properties of double vertex graphs has been widely studied since the 90's. However only very few results are know about the independence number of such graphs. In this paper we obtain the independence numbers of the double vertex graphs of fan graphs and wheel graphs. Also we obtain the independence numbers of the pair graphs, that is a generalization of the double vertex graphs, of some families of graphs. 
\end{abstract}

{\it Keywords:}  Double vertex graphs;  pair graphs; independence number.\\
{\it AMS Subject Classification Numbers:}    05C69; 05C76.

\section{Introduction.}
Let $G$ be a graph of order $n$. The double vertex graph of $G$ is defined as the graph with vertex set all $2$-subsets of $V(G)$, where two vertices are adjacent in $F_2(G)$ whenever their symmetric difference is an edge of $G$. This concept, and its generalization called $k$-token graphs, has been redefined several times and with different names. The double vertex graphs were defined and widely studied by Alavi et al.~\cite{alavi1, alavi2, alavi3}, but we can find them earlier in a thesis of G. Johns~\cite{johns}, with the name of the $2$-subgraph graph of $G$. T. Rudolph \cite{rudolph} redefined the double vertex graphs with the name of symmetric powers of graphs and used this graphs to studied the graph isomorphism problem and to study some problems in quantum mechanics and has motivated several works of different authors, see, e.g., \cite{alzaga, aude, barghi, fisch} and the references therein. Later, R. Fabila-Monroy, et. al. \cite{FFHH} reintroduce this concept but now with the name of token graphs, where the double vertex graphs are precisely the $2$-token graphs, and studied several combinatorial properties of this graphs such as: connectivity, diameter, cliques, chromatic number and Hamiltonian paths. After this work, there are a lot of results about different combinatorial parameters of token graphs, see for example \cite{dealba, beaula, token2, deepa, gomez, rive-tru, leatrujillo}. 

In H. de Alba, et. al. \cite{alba} began the study of the independence number of $k$-token graphs and in particular for the double vertex graphs of some special graphs such as: paths, cycles, complete bipartite graphs, star graphs,  etc. A subset $I$ of vertices of $G$ is an \emph{independent set} if no two vertices in $I$ are adjacent. The \emph{independence number} $\alpha(G)$ of  $G$ is the number of vertices in a largest independent set in $G$. it is know that to determine the independence number is an NP-hard~\cite{karp} problem in its generality.

In this work we obtain the independence number of the double vertex graphs of fan graphs and wheel graphs.  The {\it fan graph} $F_{m, 1}$ is defined as the join graph $P_m+K_1$, where $P_m$ denote the path graph of order $n$ and $K_1$ the complete graph of order $1$, and wheel graph $W_{m, 1}$ is defined as the joint graph $C_m+K_1$, where $C_m$ denote the cycle of order $m$. Our main results about independence number of double vertex graphs are the following:
\begin{theorem}\label{theorem01}
Let $m \geq 2$ be an integer. Then 
\[ 
\alpha\left(F_{m,1}^{(2)}\right)= \left\lfloor \frac{m^{2}}{4}\right\rfloor.
\]
\end{theorem}

\begin{theorem}\label{theorem02}
Let $m \geq 4$ be and integer. 
Then 
\[
\al\left(W_{m,1}^{(2)}\right)=\left\lfloor \frac{m}{2}  \left\lfloor   \frac{m}{2} \right\rfloor \right\rfloor.
\]
\end{theorem}

\subsection{Pair graph of graphs}
Let $G$ be a graph of order $n\geq 2$. The {\it pair graph} $C(G)$ of $G$ is the graph whose vertex set consists of all $2$-multisets of  $V(G)$ where tho vertices $\{x, y\}$ and $\{u, v\}$ are adjacent if and only if $\{x, y\}\cap \{u, v\}=\{a\}$ and if $x=u=a$, then $y$ and $v$ are adjacent in $G$. The pair graphs, also called {\it complete double vertex graph}, of a graph, were implicitly introduced by Chartrand et al. \cite{char} and defined explicitly by Jacob et. al. \cite{jacob}, were the first combinatorial properties were studied. The pair graphs are a generalization of the double vertex graphs and $G$ is always isomorphic to a subgraph of $C(G)$. 

For the case of the independence number of complete double vertex graphs we have the following results.

\begin{theorem}\label{theorem03}
If $m\geq 3$ is an integer, then
 \[
\alpha(C(P_{m}))=\left\lfloor \frac{(m+1)^{2}}{4}\right\rfloor.
 \]
\end{theorem}
\begin{theorem}\label{theorem04}
Let $m \geq 1$ be an integer. Then 
\[ 
\alpha(C(F_{m,1}))=\alpha(C(P_{m}))+1
\]
\end{theorem}
Notice that the series $\{\floor{(m+1)^{2}/4}+1\}_{m \geq 1}$ coincides with A033638$(m+1)$ in OEIS.

\begin{theorem}\label{theorem05}
Let $m \geq 3$ be an integer. Then
\[ 
\alpha(C(C_{m}))=\begin{cases} 
      k(k+1)+\floor{(k+1)/2} & m=2k+1\\
      k(k+1) & m=2k\\
  \end{cases}
\]

\end{theorem}

\begin{theorem}\label{theorem06}

Let $m \geq 3$ be an integer. Then
\[
\alpha(C(W_{m, 1}))=\alpha(C(C_m))+1.
\]
\end{theorem}

In the rest of the papers we prove all these results in different sections.

\section{Preliminary  results}

In the proofs of some of our results, we use the following known facts. 

\begin{lemma}\label{lemasubg}
If $H$ is an induced subgraph of $G$, then $\alpha(H)\leq \alpha(G)$.   
\end{lemma}

Let $G\square H$ denote the cartesian product of graphs $G$ and $H$.

\begin{proposition}\label{ind-malla}
Let $r$ and $s$ be positive integers. Then
\[
\al(P_r\square P_s)=\left\lceil \frac{r}{2}\right\rceil\left\lceil \frac{s}{2}\right\rceil+\left(r-\left\lceil \frac{r}{2}\right\rceil\right)\left(s-\left\lceil \frac{s}{2}\right\rceil \right)
\]
\end{proposition}

\begin{proposition}\label{componentes}
If $G=\cup_{i=1}^{k}G_i$, where $G_i$ is a component of $G$ with $|G_i|\geq 2$, for every $i$, then 
\[
G^{(2)}=\bigcup_{i=1} G_i^{(2)} \bigcup_{\substack{i, j=1\\ i \neq j}}^kG_{ij},
\]
where $G_{ij}\simeq G_i \square G_j$. 
\end{proposition}

The following proposition appears in the proof of Lema 12 in \cite{token2}.
\begin{proposition}\label{pborrado}
Let $X$ be a subset of $V(G)$ and $G'=G-X$. Then $F_k(G')$ is isomorphic to the graph obtained from $F_k(G)$ by deleting all vertices in $F_k(G)$ such that have al least one element of $X$.  
\end{proposition}
In~\cite{alba} was proved that $\al(P_m^{(2)})=\floor{m^2/4}$, $m\geq 2$. This is sequence $A002620(n)$ in  The On-Line Encyclopedia of Integer Sequences (OEIS)~\cite{oeis}. 
\begin{proposition}\label{propi}
Let $a(n)=A002620(n)$, $n \geq 0$.
\begin{enumerate}
\item $a(n)=\floor{n/2}\ceil{n/2}=\floor{n^2/4}$. 
\item $a(n) = a(n-1)+\floor{n/2}=a(n-1) + \ceil{(n -1)/2}$, $n > 0$, $a(0)=a(1)=0$.
\item $ a(n) = a(n-2) + n-1$,  $a(0) = 1, a(1) = 0,$ $n \geq 2$.
\end{enumerate}
\end{proposition}

\section{Proof of Theorem~\ref{theorem01}}
In $F_{m,1}$, we consider that $V(P_m)=\{1, \dots, m\}$, $E(P_m)=\{\{i, i+1\} \colon 1\leq i \leq m-1\}$ and $V(K_1)=\{m+1\}$. We use some propositions in order to prove our main result in this section. 

If $T_{m}$ is the set of all $2$-subsets of $V(P_m)$ and $B=\{\{a, m+1\}~\colon a \in V(P_m)\}$, then $\{T_m, B\}$ is a partition of $V(F_{m, 1}^{(2)})$. Notice that the subgraph of  $F_{m, 1}^{(2)}$ induced by  $T_m$ is isomorphic to $P_m^{(2)}$ and the subgraph induced by $B$ is isomorphic to $P_m$. Sometimes we use $T_m$ and $B$ as set of vertices or as the corresponding induced  subgraph.   

For  $q \in \{1, \dots, m\}$, we define the following subsets of vertices of $F_{m,1}^{(2)}$
\[
R_{q}=\lbrace\lbrace q, i\rbrace \colon i\in\lbrace 1,\ldots,m\rbrace -\lbrace q\rbrace\rbrace.
\]
In fact, $R_q \subseteq T_m$, for every $q \in \{1, \dots, m\}$ (see Figure~\ref{abanico1} for an example).
\begin{figure}[h]
\begin{center}
\includegraphics[width=.40\textwidth]{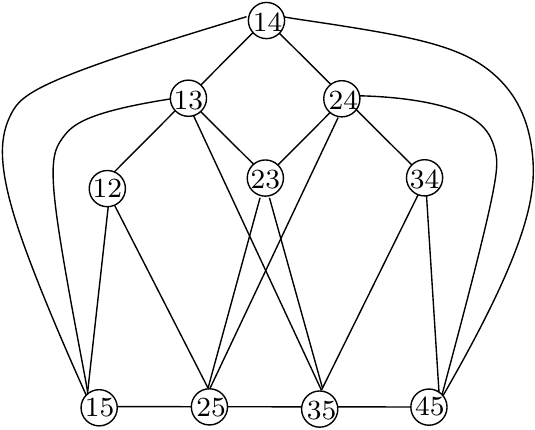}
\caption{Double vertex graph of $F_{4, 1}$. In this case $B=\{\{1, 5\}, \{2, 5\}, \{3, 5\}, \{4, 5\}\}, T_m=V(A_{4, 1}^{(2)})-B$ and $R_2=\{\{1,2\}, \{2, 3\},\{2, 4\}\}$.}
\label{abanico1}
\end{center}
\end{figure}
\begin{proposition}\label{propri}
Let $m \geq 4$ be an integer. Then $\al(T_m-R_i) = \al(T_{m-1})$, for all $i \in \{1, \dots, m\}$.
\end{proposition}
\begin{proof}
By proposition~\ref{pborrado}, $T_m-R_i$ is isomorphic  to the double vertex graph of $P_m-i$, for every $i \in \{1, \dots, m\}$. If $i \in \{1, m\}$, then $T_m-R_i$ is isomorphic to $T_{m-1}$ and hence $\al(T_m-R_i) = \al(T_{m-1})$. If $i \in \{2, m-1\}$, then $T_m-R_i$ consists of two components: one isomorphic to $P_{m-2}$ and the other isomorphic to $T_{m-2}$. Therefore
\[
\al(T_m -R_i)=\al(P_{m-2})+\al(T_{m-2})=\ceil{(m-2)/2}+\floor{(m-2)^2/4}=\floor{(m-1)^2/4}.
\]
Finally, if $i \in \{1, \dots, m\}-\{1, 2, m-1, m\}$, then $P_m-i$ consists of two components: $P_m-\{i, \dots, m\}$ and $P_m-\{1, \dots, i\}$. Then, by Proposition~\ref{componentes} it follows that the double vertex graph of $P_m-i$, that is isomorphic to $T_m-R_i$, has three components. The first one is the double vertex graph of $P_m-\{i, \dots, m\}$ and is isomorphic to $T_{i-1}$.  The second component is the double vertex graph of $P_m-\{1, \dots, i\}$ and is isomorphic to $T_{m-i}$. Finally, the third component is isomorphic to the grid graph $P_{m-i} \square P_{i-1}$. Therefore
\begin{eqnarray*}
\al(T_m -R_i)&=&\al(T_{m-i})+\al(T_{i-1})+\al(P_{m-i} \square P_{i-1}),
\end{eqnarray*}
and hence
\begin{eqnarray*}
 \lefteqn{ \al(T_m -R_i)=\floor{(m-i)^2/4}+\floor{(i-1)^2/4}+} \\ & &
\left\lceil \frac{m-i}{2}\right\rceil\left\lceil \frac{i-1}{2}\right\rceil+\left(m-i-\left\lceil \frac{m-i}{2}\right\rceil\right)\left(i-1-\left\lceil \frac{i-1}{2}\right\rceil \right)\\  &= &
\left\lfloor\frac{(m-1)^2}{4}\right\rfloor.
\end{eqnarray*}

\end{proof}
In Figure~\ref{t-borrado} (a) and (b) we show the graphs $T_8-R_2$ and $T_6-R_1$, respectively.
\begin{figure}[h]
\begin{center}
\includegraphics[width=1.08\textwidth]{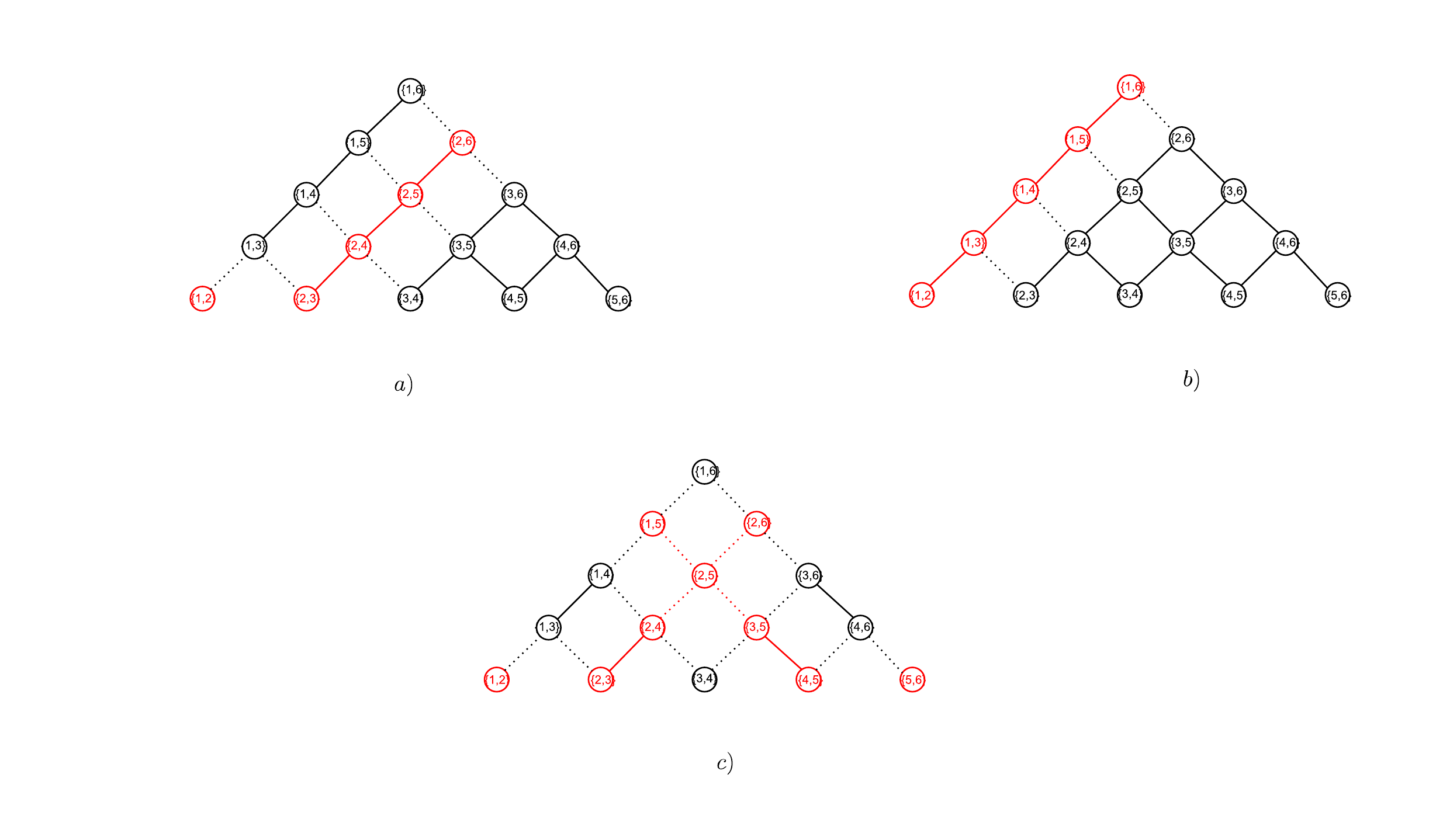}
\caption{a) Graph $T_8-R_2$, b) 
graph $T_6-R_1$, c) graph $T_6-(R_2\cup R_5)$.}
\label{t-borrado}
\end{center}
\end{figure}
\begin{proposition}\label{coroimport1}
Let $m \geq 4$ be an integer. Let $S$ be a nonempty subset of $\{1, \dots, m\}$ such that in $S$ there does not  exist consecutive integers. Then $\al(T_m-\cup_{i\in S}R_i) \leq \al(T_{m-1})$. Even more, if  $|S|\geq 2$, then $\al(T_m-\cup_{i\in S}R_i) < \al(T_{m-1})$.
\end{proposition}
\begin{proof}
If $|S|=1$ the result follows from Proposition~\ref{propri}. If $|S|\geq 2$ then $T_m-\cup_{i\in S}R_i$ is an induced subgraph of $T_m-R_x$, for every $x \in S$ and by Lemma~\ref{lemasubg} it follows that $\al(T_m-\cup_{i\in S}R_i) \leq \al(T_{m-1})$. In the view of Lemma~\ref{lemasubg}, it is enough to prove the second part of the affirmation for  $|S|=2$. The proof is by contradiction. Let $S=\{i, j\}$, with $1\leq i <j$.  Suppose that $\al(T_m-R_i \cup R_j) = \al(T_{m-1})$. Let $I$ be an independent set in $T_m-R_i \cup R_j$ of cardinality $\al(T_{m-1})$. We have two cases.

{\it Case 1.} If $j=m$, then $\{m-1, m\} \in R_j$, and hence $\{m-1, m\} \not\in I$. Also $\{m-1, m\} \not\in R_i$ because $i \neq j-1$, by hypothesis. Therefore the set $I'=I \cup\{\{m-1, m\}\}$ will be another independent set of $T_m-R_i$ of cardinality greater than  $\al(T_{m-1})$, a contradiction. 

{\it Caso 2.} $1<j<m$. Let  $X=\{j-1, j\}$ and $Z=\{j, j+1\}$. As the vertices $X$ and $Z$ belong to $R_j$, then  both vertices does not belong to $I$. The open neighborhood of $\{X, Z\}$ is a subset of $\{\{j-2, j\}, \{j-1, j+1\}, \{j, j+2\}\}$. Of this vertices, only $\{j-1, j+1\}$ could be in $I$. Therefore, the set  $I'=(I-\{ \{j-1, j+1\}\})\cup \{X, Z\}$ is and independent set in $T_m-R_i$ (because $i \not \in \{j-1, j+1\}$) of cardinality greater than $\al(T_{m-1})$, a contradiction. 
\end{proof}
In Figure~\ref{t-borrado} (c) we show the graph $T_6-(R_2\cup R_5)$.


 It is clear that $\alpha(F_{1,1}^{(2)})=1$. We are ready to prove our result. 
\begin{customthm}{\ref{theorem01}}\label{eight}
Let $m \geq 2$ be an integer. Then 
\[ 
\alpha(F_{m,1}^{(2)})= \lfloor m^{2}/4\rfloor
\]
\end{customthm}

\begin{proof}
The case $m=3$ can be checked by hand or by computer. We suppose that $m\geq 4$. As $T_m$ is isomorphic to $P_{m}^{(2)}$ it follows from Lemma \ref{lemasubg} that $\alpha(P_{m}^{(2)}) \leq \alpha(F_{m,1}^{(2)})$. We will show that $\alpha(F_{m,1}^{(2)})\leq\alpha(P_{m}^{(2)})$. 
We use the fact that the subgraph of $F_{m,1}^{(2)}$ induced by the vertex set $B$ is isomorphic to $P_m$ and $\alpha(P_m)= \lceil m/2\rceil$. 

Let $I$ be an independent set in $F_{m,1}^{(2)}$. We have the following cases: 

{\it Case 1.} If $I\subseteq T_{m}$, then $| I |\leq \alpha(P_{m}^{(2)})$. 

{\it Case 2.} If $I\subseteq B$, then $| I | \leq \lceil m/2\rceil\leq\lfloor m^{2}/4\rfloor$ because $m\geq 4$. 

{\it Case 3.} $I=I^{'}\cup I^{''}$, with $I^{'}\subset T_{m}, I^{''}\subset B$, and $I^{'}, I^{''}$ both non-empty.
Let $S=\{i \colon \{i, m+1\} \in I''\}$. As $I''$ is an independent set, in $S$ there does not exists any two consecutive integers. We claim that $\cup_{i \in S}R_i \cap I'=\emptyset$. Indeed,  suppose that $\{x, y \} \in \cup_{i \in S}R_i \cap I'$. Then $\{x, y\} \in R_i$, for some $i \in S$. Without lost of generality we can suppose that $\{x, y\}=\{i, y\}$, with $y \in \{1, \dots, m\}-\{i\}$. By definition of $S$, vertex $\{i, m+1\}$ belong to $I''$.  Now $\{i, y\} \triangle \{i, m+1\}=\{y, m+1\}$ that is an edge in $F_{m,1}$, and hence $\{i, y\} \sim  \{i, m+1\}$ in $F_{m,1}^{(2)}$. But this is a contradiction, since $I$ is an independent set. This shows that $I'$ is a subset of $T_m-\cup_{i \in S}R_i$ and hence 
\begin{equation}\label{ecuacion1}
 |I'|\leq \al(T_m-\cup_{i \in S}R_i).
\end{equation}
Now we show that $|I|\leq  \lfloor m^{2}/4\rfloor$. 
 If $|I''|=1$, then $|S|=1$ and by Proposition~\ref{propri}
\[
|I| =  \al(T_{m-1})+1\leq \al(T_{m-1})+\lfloor m/2\rfloor= \lfloor m^{2}/4\rfloor,
\]
where the last inequality follows from Proposition~\ref{propi}(2). Finally, consider that $2\leq |I^{''}|\leq \lceil m/2\rceil$. As $|I''|\geq 2$, then $|S|\geq 2$. By Proposition~\ref{coroimport1} and Equation~(\ref{ecuacion1}) we have that  que

\begin{eqnarray*}
\mid I\mid &\leq& \alpha(T_m-U_{i \in S}R_i)+|I''| \\
&<& \alpha(T_{m-1})+\lceil m/2 \rceil\\
&\leq&\alpha(T_{m-1})+\lceil m/2 \rceil-1\\
&\leq&\alpha(T_{m-1})+\lfloor m/2 \rfloor\\
&\leq&\lfloor m^{2}/4\rfloor.
\end{eqnarray*}

\end{proof}
\section{Proof of Theorem~\ref{theorem02}}

For the wheel graph $W_{m, 1}$ we consider that $V(C_m)=\{1, \dots, m\}$, $E(C_m)=\{\{i, i+1\} \colon 1\leq i \leq m-1\} \cup \{\{1, m\}\}$ and $V(K_1)=\{m+1\}$. Let $H_m$ denote the subgraph of $W_{m,1}^{(2)}$ induced by all $2$-subsets of $V(C_m)$. Let $D$ denote  the subgraph of $W_{m,1}^{(2)}$  induced by the vertex set $\{\{i, m+1\}\colon i \in \{1, \dots, m\}\}$. The graph $H_m$ is isomorphic to $C_m^{(2)}$ and $D$ is isomorphic to $C_m$. We also use $H_m$ and $D$ as vertex sets.  It is well-known that $\al(C_m)=\floor{m/2}$ and in \cite{alba} was proved that $\al(C_m^{(2)})=\floor{m\floor{m/2}/2}$, $m \geq 3$.

It can be checked by computer that $\al(W_{3,1}^{(2)})=2$. We now prove our main result in this section.
\begin{customthm}{\ref{theorem02}}\label{eight}
Let $m \geq 4$ be and integer. Then 
\[\al(W_{m,1}^{(2)})=\al(C_m^{(2)}).
\]
\end{customthm}
\begin{proof}
As $H_m$ is isomorphic to $C_m^{(2)}$, then every independent set $I$ in $H_m$ satisfies  
$|I| \leq \al(C_m^{(2)})\leq \al(W_{m,1}^{(2)})$. We will show that $\al(W_{m,1}^{(2)})\leq  \al(C_m^{(2)})$ .

Let $I$ be an independent set in $W_{m,1}^{(2)}$. If $I \subset H_m$, then $|I|\leq \al(C_m^{(2)})$. If $I\subset D$, then 
\[
|I| \leq \floor{m/2} \leq \floor{m\floor{m/2}/2}=\al(C_m^{(2)}).
\]

Now suppose that $I=I'\cup I''$, where $I'\subset H_m$, $I''\subset D$, and with $I', I''$ both non- empty sets. As $D$ is isomorphic to $C_m$ then $|I''| \leq \floor{m/2}$. For $q \in \{1, \dots, m\}$, let $R_q$ defined as in previous section. Let 
 \[
 U=\{i \in \{1, \dots, m\}\colon \{i, m+1\} \in I''\}.
 \]
 In a similar way that in the proof of Theorem~\ref{theorem01}, it can be showed that $\cup_{i \in U}R_i \cap I'=\emptyset$.  

Therefore we have $|I'|\leq \al(H_m-\cup_{i \in U}R_i)$.

 By Proposition~\ref{pborrado}, the graph $H_m-R_q$ is isomorphic to the double vertex graph of $C_m-q$, for every $q \in \{1, \dots, m\}$. But as $C_m-q$ is isomorphic to $P_{m-1}$ then $H_m-R_q \simeq T_{m-1}$.
 We like to bound $|I|$ using that $|I'| \leq \al(H_m-\cup_{i \in U}R_i)$. First, consider that $U=\{x\}$, for some $x \in \{1, \dots, m\}$, that is $I''=\{\{x, m+1\}\}$. By the previous paragraphs we have that 
 \[
|I| \leq \al(H_m-R_x)+1=\al(T_{m-1})+1= \floor{(m-1)^2/4}+1\leq  \floor{m\floor{m/2}/2},
\]
where the last inequality holds because $m\geq 4$.

 Now, suppose that $|U|\geq 2$. First note that if $q \in V(C_m)$, then the double vertex graph of $W_{m, 1}-q$ is isomorphic to the double vertex graph of $F_{m-1,1}$. Therefore, by Proposition~\ref{pborrado} it follows that 
\begin{equation}\label{iso2}
 W_{m, 1}^{(2)}-(R_q \cup \{\{q, m+1\}\} \simeq (W_{m, 1}-q)^{(2)} \simeq F_{m-1,1}^{(2)}.
\end{equation}
We like to obtain $\alpha(H_m - \cup_{i \in U}R_i)$. By Equation~(\ref{iso2}), after we delete one set $R_q$ and the vertex $\{q, m+1\}$ from $W_{m, 1}^{(2)}$, for $q \in U$, we obtain an isomorphic copy of $F_{m-1,1}^{(2)}$, which in turn contains isomorphic copies of the remaining sets $R_i$, for $i \in U -\{q\}$. Then we are in Case (3) of the proof of Theorem~\ref{theorem01} for $F_{m-1,1}^{(2)}$ and $S=U-\{x\}$ (with the corresponding relabeling given by the isomorphism between  $(W_{m, 1}-q)^{(2)}$ and $F_{m-1,1}$). Therefore, for any $x \in U$
 \[
 H_m-\cup_{i \in U}R_i=(H_m-R_x)-\cup_{i \in U-\{x\}}R_i\simeq T_{m-1}-\cup_{i \in S}R_i,
 \]
Using Proposition~\ref{coroimport1} for $F_{m-1,1}^{(2)}$ and $S$ we have that  
\[
\al(H_m-\cup_{i \in U}R_i)=\al(T_{m-1}-\cup_{i \in S}R_i)\leq \al(T_{m-2}).
\]

And hence

\begin{eqnarray*}
|I| &\leq& \al(H_m-\cup_{i \in U}R_i)+|I''|\\
&\leq & \al(T_{m-2})+\floor{m/2}\\
&\leq& \floor{(m-2)^2/4}+\floor{m/2}\\
&\leq& \floor{m\floor{m/2}/2}\\
&=&\al(C_m^{(2)}).
\end{eqnarray*}
\end{proof}

\section{Proof of Theorem~\ref{theorem03}}

The proof of Theorem~\ref{theorem03} follows directly from the following result.  
\begin{theorem}\label{iso-cyd}
For any non negative integer $n\geq 3$ we have  
\[C(P_{n})\simeq P_{n+1}^{(2)}\]
\end{theorem}
\begin{proof}
 Without loss of generality we can suppose that for $\{a, b\} \in V(C(P_n))$, $a\leq b$, and for $\{a, b\}$ in $P_m^{(2)}$, $a<b$. Let $\phi\colon C(P_{n})\to P_{n+1}^{(2)}$ be the function given by $\phi(\{i,j\})=\{i,j+1\}$. It is an exercise to show that this function is a graph isomorphism between $C(P_{n})$ and $P_{n+1}^{(2)}$. 
 \end{proof}
 
\section{Proof of Theorem~\ref{theorem04}}
The vertex set of $C(F_{m, 1})$ can be partitioned in  $\{T_{m+1}, B\}$ where $T_{m+1}$ (as induced graph) is isomorphic to $C(P_m)$ (that is isomorphic to $P_{m+1}^{(2)}$ by Theorem~\ref{iso-cyd}) and 
\[
B=\{\{i, m+1\} : 1\leq i \leq m+1\}.
\]

Notice that the subgraph of $C(F_{m, 1})$ induced by $B$ is isomorphic to $F_{m, 1}$.

We define the following subset of vertices of  $C(F_{m,1})$.
\[
R_{i}=\lbrace\lbrace i, j\rbrace \colon j\in\lbrace 1,\ldots,m\rbrace \rbrace,
\]

The following proposition will be useful in the proof of Theorem~\ref{theorem04}
\begin{proposition}\label{procompleta}
For $m \geq 2$, we have that  $\al(T_{m+1}-R_i) \leq \floor{m^2/4}+1$, for any $i \in \{1, \dots, m\}$.
\end{proposition}
\begin{proof}
Notice that for $i \in \{1, \dots, m\}$, the graph $T_{m+1}-R_i$ is isomorphic to the graph $C(P_m-i)$. We have several cases.

{\it Case 1.} If $i \in \{1, m\}$, then the graph $T_{m+1}-R_i$ is isomorphic to $C(P_{m-1})$, that it is isomorphic to $P_m^{(2)}$ (by Theorem~\ref{iso-cyd}), and hence $\al(T_{m+1}-R_i) =\floor{m^2/4}$. 

{\it Case 2.} If $i \in \{2, m-1\}$, then $T_{m+1}-R_i$ consists of three components as follows. One component that is isomorphic to $K_1$. Such component $K_1$ is either the vertex $\{1, 1\}$, or the vertex $\{m, m\}$ if $i=2$ or $i=m-1$, respectively. Another component consists either, in the subgraph generated by the vertices $R_1-\{\{1, 2\}, \{1, 1\}\}$, when  $i=2$, or $R_m-\{\{m-1, m\}, \{m, m\}\}$ when $i=m-1$. This component is isomorphic to $P_{m-2}$.  The last component is isomorphic to $T_{m-1}$: when $i=2$, $T_{m-1}$ will be the subgraph generated by the set of vertices  $T_{m+1}-(R_1 \cup R_2)$, and when $i=m-1$, $T_{m-1}$ will be the subgraph generated by the set of vertices $T_{m+1}-(R_{m-1} \cup R_m)$. Then
\begin{eqnarray*}
\al(T_{m+1} -R_i)&=&\al(T_{m-1})+\al(P_{m-2})+1\\
&=&\floor{(m-1)^2/4}+\ceil{(m-2)/2}+1\\
&\leq&\floor{(m-1)^2/4}+\ceil{(m-1)/2}+1\\
&\leq& \floor{m^2/4}+1,
\end{eqnarray*}
where, for the last inequality, we use the fact that $\ceil{(m-2)/2} \leq \ceil{(m-1)/2}$ and part 2 of Proposition~\ref{propi}. 

{\it Case 3.} If $i \in \{1, \dots, m\}-\{1, 2, m-1, m\}$, then $T_{m+1}-R_i$ consists of three components that came from the double vertex graph of $P_m-i$. The first component is isomorphic to $C(P_m-\{1, \dots, i\})$, that in fact  is isomorphic  to $T_{m-i+1}$. The other component is isomorphic to $C(P_m-\{i, \dots, m\})$ that is isomorphic to $C(P_{i-1})$, which in turn is isomorphic to $T_{i}$. The last component of $T_{m+1}-R_i$ is the subgraph of  $T_{m+1}$ induced for the set of vertices of the form $\{a, b\}$ with $a \in \{1, \dots, i-1\}$ and $b \in \{i+1, \dots, m\}$. This last component is isomorphic to the grid graph $P_{m-i} \times P_{i-1}$. Therefore
\begin{eqnarray*}
\al(T_{m+1} -R_i)&=&\al(T_{m-i+1})+\al(T_{i})+\al(P_{m-i} \times P_{i-1}).
\end{eqnarray*}
Therefore
\begin{eqnarray*}
 \lefteqn{ \al(T_{m+1} -R_i)=\floor{(m-i+1)^2/4}+\floor{i^2/4}+} \\ & &
\left\lceil \frac{m-i}{2}\right\rceil\left\lceil \frac{i-1}{2}\right\rceil+\left(m-i-\left\lceil \frac{m-i}{2}\right\rceil\right)\left(i-1-\left\lceil \frac{i-1}{2}\right\rceil \right)\\  
&\leq & \left\lfloor\frac{m^2}{4}\right\rfloor+1.
\end{eqnarray*}
\end{proof}

\begin{corollary}\label{coroimport2}
$\al(T_{m+1}-\cup_{i\in S}R_i) \leq \floor{m^2/4}+1$, for every $S \subseteq \{1, \dots, m\}$, $S \neq \emptyset$. 
\end{corollary}
\begin{proof}
For every $j \in S$, $T_{m+1}-\cup_{i\in S}R_i$ is an induced subgraph of $T_{m+1}-R_j$ and the result follows by Lemma~\ref{lemasubg} and by Proposition~\ref{procompleta}.
\end{proof}
In Figure~\ref{completa2} we show graph $T_{9+1} -R_2 \cup R_5$.
\begin{figure}[h]
\begin{center}
\includegraphics[width=.60\textwidth]{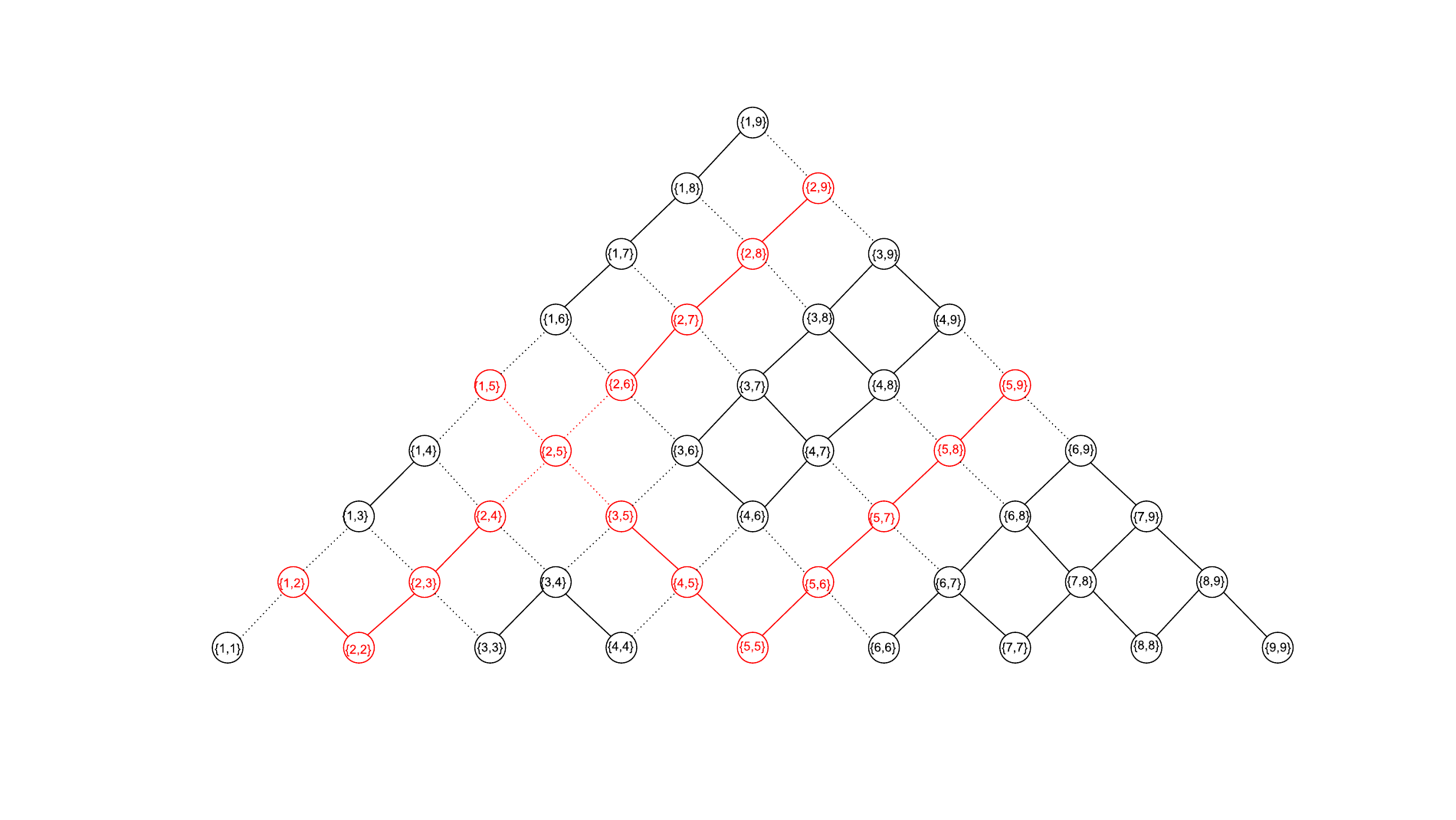}
\caption{Graph $T_{10} -R_2 \cup R_5$}
\label{completa2}
\end{center}
\end{figure}

\begin{customthm}{\ref{theorem04}}
Let $m \geq 3$ be an integer. Then 
\[ 
\alpha(C(F_{m,1}))=\alpha(C(P_{m}))+1
\]
\end{customthm}
\begin{proof}
Let $I$ be an independent set in $T_{m+1}$ of cardinality $|I|=\alpha(C(P_{m}))$. As $I \cup \{ \{m+1, m+1\}\}$ is an independent set in $C(F_{m, 1})$ then $\al(C(F_{m, 1}))\geq \alpha(C(P_{m}))+1$.

We will show that $\al(C(F_{m, 1}))\leq \alpha(C(P_{m}))+1$. The case $m= 3$ is easy so that we suppose that $m\geq 4$. Let $I$ be an independent set  in $C(F_{m,1})$. We have several cases.

{\it Case 1.} If $I\subseteq T_{m+1}$, then $| I |\leq \alpha(C(P_{m}))$. 

{\it Case 2.} If $I\subseteq B$, then $| I | \leq \lfloor (m+1)/2\rfloor \leq\lfloor (m+1)^{2}/4\rfloor$, because $B\simeq F_{1, m}$ and $m \geq 4$.

{\it Case 3.} If $I=I^{'}\cup I^{''}$, such that $I^{'}\subset T_{m+1}, I^{''}\subset B$, with $I^{'}$ e $I^{''}$ non empty sets. Let $S=\{i \colon \{i, m+1\} \in I''\}$. In a similar way  that in the proof of  Case 3 of Theorem~\ref{theorem01} we can show that $\cup_{i \in S}R_i \cap I'=\emptyset$. This shows that $I'$ is a  subset of $T_ {m+1}-\cup_{i \in S}R_i$. By corollary~\ref{coroimport2} we have that $|I'|\leq \floor{m^2/4}+1$ and as $B\simeq C_{m+1}$ then $|I''|\leq \floor{(m+1)/2}$. Therefore
\begin{eqnarray*}
|I|&\leq& \floor{m^2/4}+\floor{(m+1)/2}+1\\
&=& \floor{(m+1)^2/4}+1\\
&=&\alpha(C(P_m))+1, 
\end{eqnarray*}
where we are using part (2) of Proposition~\ref{propi}.
\end{proof}

\section{Proof of Theorem~\ref{theorem05}}

We proof Theorem~\ref{theorem05} by mean of Propositions~\ref{keven} and \ref{kodd}. 
For $q=1,\ldots,m$, let
\[
L_{q}:=\lbrace\lbrace j,m-(q-j)\rbrace\colon1\leq j\leq q\rbrace.
\]
It is clear that $|L_q|=q$, for every $q$, and that $\{L_1 \dots, L_q\}$ is a partition of $V(C(C_{m}))$. The following proposition shows that most of the sets $L_q$ are independent sets in $C(C_{m})$.
\begin{proposition}\label{lqciclo}
If $L_{q}$ is not an independent set in $C(C_{m})$, then $m=2q-1$, where $2\leq q \leq m-1$. 
\end{proposition}
\begin{proof}
Clearly $L_1=\{\{1, m\}\}$ and $L_m=\{\{1, 1\}, \{2, 2\}, \dots, \{m, m\}\}$ are independent sets and hence  $2\leq q \leq m-1$. As $L_q$ is not an independent set and $q \geq 2$, then there exists two adjacent vertices, say $\{i,m-(q-i)\}$ and $\{j, m-(q-j)\}$, in $L_q$. Notice that $i \neq j$ and $i \neq m-(q-i)$. Therefore $i =m-(q-j)$ and $|m-(q-i)-j| \in \{1, m-1\}$. From these equations we obtain that $|m-(q-i)-j|=|2(m-q)|$, which implies that $m=2q-1$. 
\end{proof}
Let $G$ be a graph and let $A$ and $B$ subsets of  $V(G)$.  We say that $A$ and $B$ are linked in $G$, and is denoted by $A \approx B$,  if $G$ has an edge $ab$ such that $a \in A$ and $b \in B$.

\begin{proposition}\label{links} Let $m\geq 4$. The subsets $L_{i}$ of $ V(C(C_m))$ previously defined are linked as follows: 
\begin{enumerate}
\item $L_{i}\approx L_{i+1}$, for $i\in \lbrace1,\ldots,m-1\rbrace$.
\item $L_{i}\approx L_{m-i+1}$, for $i\in\lbrace 1,\ldots,m-1\rbrace$.
\item All the links between the elements in $\lbrace L_{1}\ldots,L_{m}\rbrace$ are given by (1) and (2).
\end{enumerate}
\end{proposition}
\begin{proof}
(1) For $1\leq i \leq m=1$, the sets $L_i$ and $L_{i+1}$ are linked because $\{i, m\} \in L_i, \{i+1,m\} \in L_{i+1}$ and $[\{i, m\}, \{i+1,m\}]$ is an edge in $C(C_m)$. \\
(2) By definition of $L_q$ it follows that $\{\{i, m\}$ and $\{1, m-(i-1)\}$ belongs to $L_i$, and the vertices $\{1, i\}$ and $\{m-(i-1), m\}$ belongs to $L_{m-(i-1)}$. As $[\{\{i, m\}, \{1, i\}]$ and $[\{1, m-(i-1)\}, \{m-(i-1), m\}]$ are edges in $C(C_m)$ we obtain that $L_i \approx L_{m-i+1}$.\\
(3) If $L_i$ is linked with a set $L_j$, with $|i-j| \neq 1$, then, by the construction of $C(C_m)$, the unique possible vertices in $L_i$ that could be adjacent with vertices in $L_j$ are $\{1, m-i+1\}$ and $\{i, m\}$. But this would implies that $j=m-i+1$ and we are in Case 2.
\end{proof}

\begin{proposition}\label{keven}
Let $k\geq 2$ be an integer. Then
\[
\alpha(C(C_{2k}))=k(k+1).
\] 
\end{proposition}
\begin{proof}
By Propositions~\ref{lqciclo} and~\ref{links} we have that 
\[
I=L_2 \cup L_4 \cup \dots \cup L_{m-2} \cup L_m 
\]
is an independent set in $C(C_m))$.
Now
\begin{eqnarray*}
|I|&=&|L_2|+|L_4|+\dots + |L_{2k-2}|+|L_{2k}|\\
&=&2+4+\dots + 2k-2+2k\\
&=&k(k+1).
\end{eqnarray*}
Therefore $\alpha(C(C_{2k}))\geq k(k+1)$.

Now, as $C(P_m)$ is a subgraph of $C(C_m)$ with $V(C(P_m))=V(C(C_m))$ and $E(C(P_m))\subset E(C(C_m))$, then $
\alpha(C(C_m))\leq \alpha (C(P_m))$. Using Theorem~\ref{iso-cyd} we obtain
\begin{eqnarray*}
\alpha(C(C_{2k}))&\leq& \alpha (C(P_{2k}))\\
&=& \alpha(P_{2k+1}^{(2)}))\\
&=& \floor{(2k+1)^2/4}\\
&=&  k(k+1).
\end{eqnarray*}
\end{proof}

The following proposition will be useful.

\begin{proposition}\label{not1n}
Let $n=2k+1$ be an odd positive integer. Let $I$ be and independent set of $C(C_{n})$. If the vertex $\{1, n\}$ belongs to $I$, then there exist an independent set $I'$ such that $\{1, n\} \not\in I'$ and $|I'|\geq |I|$.
\end{proposition}
\begin{proof}
Let $I=I_1 \cup I_2$, where $I_2=L_n \cap I$ and $I_1=I-I_2$. If $\{1, n\} \in I$, then the vertices $\{1, 1\}$ and $\{n, n\}$ does not belongs to $I$. Let $m=|L_{n-1} \cap I|$. Then, there are at least $m$ vertices in $L_n-\{\{1, 1\}, \{n, n\}\}$ such that does not belongs to $I$. That is, $|L_n \cap I|\leq n-2-m$, and hence $|I|\leq |I_1|+n-2-m$.  We construct $I'=I'_1\cup I'_2$ as follows,   $I'_1=I_1-(L_{n-1}\cap I) \cup \{\{1, n\}\}$ and $I'_2=L_n$. Clearly $I'$ is an independent set.  Therefore 
\[
|I'|=|I'_1|+|I'_2|=|I_1|-m-1+n\geq |I|.
\]
\end{proof}
If $x$ is a vertex in $V(C(C_n))$ of the form $\{1, j\}$ or $\{j, n\}$, for some $j$, then $x$ is called an {\it extreme vertex} in $C(C_n)$. Let $\{i, j\} \in V(C(C_n))$, with $i<j$. We say that $x \in \{\{i, j+1\}, \{i+1, j\}\}$ (resp. $x \in \{\{i-1, j\}, \{i, j-1\}\}$) is a {\it right neighbor} of $\{i, j\}$ (resp. {\it left neighbor}) if $x$ is adjacent to $\{i,j\}$ in $C(C_n)$.   

\begin{proposition}\label{kodd}
Let $k\geq 1$ be an integer. Then
\[
\alpha(C(C_{2k+1}))=k^2+k+\floor{(k+1)/2}.
\] 
\end{proposition}
\begin{proof}
The case $k=1$ is easy so we assume that $k\geq 2$. \\

{\bf Case $k$ odd}. Let 
\[
L=L_2\cup L_4 \cup \dots \cup L_{k-1} \cup L_{k+2} \cup L_{k+4} \dots \cup L_{2k+1},
\]
which is an independent set of $C(C_{2k+1})$ (by Propositions~\ref{lqciclo} and~\ref{links}).
We have that 
 \begin{eqnarray*}
|L|&=&2+4 +\dots+(k-1)+(k+2)+(k+4)+\dots+(2k+1)\\
&=&k^2+\frac{3}{2}k+\frac{1}{2}\\
&=&k^2+k+\left\lfloor \frac{k+1}{2} \right\rfloor,
\end{eqnarray*}
which shows that $\alpha(C(C_{2k+1}))\geq k^2+k+\floor{(k+1)/2}$. 

Now we prove that $\alpha(C(C_{2k+1})) \leq  k^2+k+\left\lfloor \frac{k+1}{2} \right\rfloor$. 

Let $n=2k+1$. Let $A=L_1 \cup L_2 \cup \dots \cup L_{k}$ and $B=L_{k+1} \cup L_{k+2} \cup \dots \cup L_{n}$.
Let $I$ be and independent set in $C(C_{2k+1})$. By Proposition~\ref{not1n} we can assume that $\{1, n\} \not \in I$. Let  $I_1=I \cap A$ and $I_2=I \cap B$. 

We will construct and independent set $I'$ of $C(C_{2k+1})$ such that $|I'| \geq |I|$ and $I'\subseteq L$. First we list the steps and then we say how to do each step.
\begin{itemize}
\item[Step 1.]To construct an independent set $I_2'$ from $I_2$ such that $|I_2'|\geq |I_2|$ and 
\[
I_2' \subseteq L_{k+2} \cup L_{k+4} \cup \dots \cup L_{2k+1}
\]
\item[Step 2.] To construct a set $I''$ from $I_1\cup I_2'$ by interchanging every vertex $\{i, n\}\in I\cap (L_1 \cup L_3 \cup \dots \cup L_k)$ (if any) with $\{1, i\}$. 
\item[Step 3.] To construct $I'$ from $I''$ by interchanging every vertex in $I'' \cap L_3 \cup L_5 \cup \dots \cup L_{k}$ with a vertex in $L_2 \cup L_4 \cup \dots \cup L_{k-1}$ in such a way that $I'\subseteq L$.
\end{itemize}
The set $I_2'$ in Step 1 is obtained in the following  way:

Let $W=\{k+1, k+3, \dots, 2k\}$. Let $i_1$ be the greatest integer in $W$ such that $ L_{i_1} \cap I_2\neq \emptyset$. That is, $L_w \cap I_2=\emptyset$, for every $w \in W$ with $w >i_1$.  To construct $J_1$ from $I_2$ by interchanging every vertex in $L_{i_1} \cap I_2$ with its right neighbor in $L_{i_1+1}$. Now, Let $i_2$ be the greatest integer in $W-\{i_1\}$ such that $ L_{i_2} \cap J_1\neq \emptyset$.  To construct $J_2$ from $J_1$ by interchanging every vertex in $L_{i_2} \cap J_1$ with its right neighbor in $L_{i_2+1}$. We will continue this procedure until we find a set $J_{r}$ such that $W-\{i_1, \dots, i_r\}=\emptyset$ or $L_w \cap J_{r} =\emptyset$, for any $w \in W-\{i_1, \dots, i_r\}$. That is, we will finish until all the vertices in $I_2 \cap \left(L_{k+1} \cup L_{k+3} \cup \dots \cup L_{2k}\right)$ has been interchanged with vertices in $L_{k+2} \cup L_{k+4} \cup \dots \cup L_{2k+1}$.  We obtain the desired independent set by making $I_2'=J_{i_r}$. \\

Let $X=L_1 \cup L_3 \cup \dots \cup L_k$. Notice that Step 2 can be done because if $\{i, n\} \in I_1\cap X$, then $\{1, i\} \not \in I_2'$. We have that  $\{1, i\} \in L_{k+2} \cup L_{k+4} \cup \dots \cup L_{2k+1}$ because $\{i, n\} \in X$ and hence $I''\cap B \subseteq L_{k+2} \cup L_{k+4} \cup \dots \cup L_{2k+1}
$.  
 
Finally we show how to realize Step 3.  Let $I_1'=I''\cap X$. First, we construct $M_1$ from $I''$ by   interchange all the vertices in $I_1'$ as follows: to select the smallest integer $i$ in $\{3, 5, \dots, k\}$ such that $I_1'\cap L_i \neq \emptyset$ and $I_1'\cap L_j=\emptyset$, for every $j \in \{3, 5, \dots, k\}$ with $j<i$. To obtain $M_1$ from $I''$ by interchange every vertex in $I_1'\cap L_{i}$ with its respective right neighbor in $L_{i-1}$. To repeat this process to obtain $M_2$ from $M_1$ but now with the smallest integer in $\{3, 5, \dots, k\}-\{i_1\}$ such that $I_1'\cap L_i \neq \emptyset$ and $I_1'\cap L_j=\emptyset$, for every $j \in \{3, 5, \dots, k\}-\{i_1\}$ with $j<i$. To continue in this way until a set $M_r$ is obtained in which all the vertices in $I_1'\cap X$ has been interchanged by vertices in  $L_2 \cup L_4 \cup \dots \cup L_{k-1}$. Now to make $I'=M_s$.  
 Notice that  $|I'|\geq |I|$ and $I'\subseteq L$  which implies that $|I|\leq |L|$ as desired.

{\bf Case $k$ even.}

First we show that $\alpha(C(C_{2k+1}))\geq k^2+k+\floor{(k+1)/2}$.
Let
\[
L=L_2\cup L_4 \cup \dots \cup L_k \cup L_{k+3} \cup L_{k+5}\cup \dots \cup L_{2k+1},
\]
which is an independent set of $C(C_{2k+1})$.
We have that
\begin{eqnarray*}
|L|&=&2+4+ \dots + k + (k+3)+(k+5)+\dots +(2k+1)\\
&=&\frac{1}{4} k (k+2) + \frac{1}{4} (4 k + 3 k^2)\\
&=&k^2+\frac{3}{2}k=k^2+k+\left\lfloor \frac{k+1}{2} \right\rfloor.
\end{eqnarray*}

Now we prove that $\alpha(C(C_{2k+1}))\leq k^2+k+\floor{(k+1)/2}$.

Let $I$ be any independent set of $C(C_{2k+1})$. By Proposition~\ref{not1n} we can assume that $\{1, n\} \not \in I$. We will obtain an independent set $I'$ such that $I'\subseteq L$ and $|I| \leq |I'|$. Let $I_1=I \cap \left(L_2 \cup L_3 \cup \dots \cup L_{k+1}\right)$ and $I_2=I \cap \left(L_{k+2} \cup L_{k+3} \cup \dots \cup L_{2k+1}\right)$. 

Now we obtain $I'$ with the following steps.
\begin{itemize}
\item[Step 1.]  To obtain $I_2'$ from $I_2$ by interchanging every vertex in $I_2\cap L_i $ with its right neighbor in $L_{i+1}$, for every $i \in \{k+2, k+4, \dots, 2k\}$, in the same way Step 1 for the case $k$ odd.
\item[Step 2.]  To obtain $I''$ from $I_1\cup I_2'$ by interchanging all the vertices of the form $\{a, n\} \in I \cap L_i$ with $i \in \{3, 5, \dots, k-1\}$ (if any) with $\{1, a\}$.
\item[Step 3.] To obtain $I'''$ from $I''$ by interchange every vertex  in $I'' \cap L_i$, for $i \in \{3, 5, \dots, k-1\}$ with its right neighbor in $L_{i-1}$ in the same way that in Step 3 of the case $k$ odd.
\item[Step 4.] To obtain $I'$ from $I'''$ by interchange the vertices in $I \cap L_{k+1}$ as follows: as $L_{k+1}$ is linked with $L_{k+1}$ by the edge $[\{1, k+1\}, \{k+1, n\}]$, then only one of this vertices could be in $I$. If $\{\{1, k+1\}, \{k+1, n\}\} \cap I =\emptyset$, then move all the vertices in $I'''\cap L_{k+1}$ to its right neighbors in $L_k$. If $\{1, k+1\} \in I$, then move all the vertices in $I'''\cap L_{k+1}$ to its right neighbors in $L_k$, and if $\{k+1, n\} \in I$, then  move all the vertices in $I''\cap L_{k+1}$ to its left neighbors in $L_k$.    
\end{itemize}
\end{proof}

\section{Proof of Theorem~\ref{theorem06}}
In this proof, the vertices of $C_m$ in $W_{m, 1}$ are $1, \dots, m$ and the vertex in $K_1$ is $m+1$. We use $U_{m}$ to denote the subgraph $C(C_m)$ of $C(W_{m, 1})$. The sets $R_i$ are defined as in the proof of Theorem 1.4. Let 
\[
B=\{\{i, m+1\} : 1\leq i \leq m+1\}.
\]
Notice that the subgraph of $C(W_{m, 1})$ induced by $B$ is isomorphic to $W_{m,1}$ that has $\alpha(W_{m, 1})=\floor{m/2}$. We use indistinctly $B$ as set of vertices or as the subgraph of $W_{m, 1}$ induced by this set of vertices.   
\begin{proposition}\label{procompleta2}
Let $m \geq 2$ be an integer. If $S$ is a subset of $\{1, \dots, m\}$, then $\al(U_{m}-\cup_{i \in S}R_i) \leq \floor{m^2/4}$.
\end{proposition}
\begin{proof}
We known that $U_{m}-x\simeq C(P_{m-1})\simeq P_m^{(2)}$, for any $x\in S$. Therefore
\[
\alpha(U_{m}-x)=\alpha(P_m^{(2)})=\floor{m^2/4}
\]
and the result follows because  $U_{m}-\cup_{i \in S}R_i$ is an induced subgraph of $U_{m+1}-R_x$, for any $x \in S$.
\end{proof}
\begin{customthm}{\ref{theorem06}}

Let $m \geq 3$. Then
\[
\alpha(C(W_{m, 1}))=\alpha(C(C_m))+1.
\]
\end{customthm}

\begin{proof} Let $I$ be any maximal independent set of $U_m$. As $I \cup \{m+1, m+1\}$ is and independent set in  $C(W_{m,1})$ then $\alpha(C(W_{m, 1}))\geq \alpha(C(C_m))+1$. Now we prove that $\alpha(C(W_{m, 1}))\leq \alpha(C(C_m))+1$. Let $I$ be an independent set in $C(W_{m, 1})$. We have several cases: If $I \subset U_m$ then $|I| \leq \alpha(C(C_m))$. If $I \subset B$ then $|I|\leq \floor{m/2} \leq \alpha(C(C_m))$. The last case is when $I=I'\cup I''$, where $I'\subset U_m$, $I''\subset B$, and with $I', I''$ both non- empty sets. As $B$ is isomorphic to $W_{m,1}$ then $|I''| \leq \floor{m/2}$. Let 
 \[
 S=\{i \in \{1, \dots, m\}\colon \{i, m+1\} \in I''\}.
 \]
 
First consider the case when $|S|=1$. By Proposition~\ref{procompleta2} we have that 
\[
|I|\leq \floor{m^2/4}+1.
\]
If $m=2k$, then 
\[
|I|\leq \floor{m^2/4}+1\leq  \floor{k^2}+1\leq k(k+1)+1=\alpha(C(C_{2k})+1.
\]
If $m=2k+1$, then
\[
\floor{m^2/4}+1\leq \floor{k^2+k+1/4}+1 \leq k(k+1)+\floor{(k+1)/2}+1=\alpha(C(C_{2k+1})+1.
\]
When $|S|\geq 2$ it can be shown, in a similar way that for the case of the double vertex graph of the wheel graph, that
\[
\alpha(U_m -\cup_{i \in S} R_i)\leq \floor{(m-1)^2/4}.
\]
 Therefore\begin{eqnarray*}
|I|&\leq& \floor{(m-1)^2/4}+\floor{m/2} 
\end{eqnarray*}
If $m=2k$, then
\begin{eqnarray*}
|I|&\leq& \floor{(2k-1)^2/4}+\floor{2k/2} \leq k(k+1)+1.
\end{eqnarray*}
If $m=2k+1$, then
\begin{eqnarray*}
|I|&\leq& \floor{(2k)^2/4}+\floor{(2k+1)/2} \leq k^2+k+\left\lfloor \frac{k+1}{2}\right\rfloor+1,
\end{eqnarray*}
and the proof is completed. 
\end{proof}

{\sc Unidad Acad\'emica de Matem\'aticas,}\\
{\sc Universidad Aut\'onoma de Zacatecas,}\\
 {\sc Zac., Mexico.}\\
 {\it E-mail address: paloma\_101293@hotmail.com}\\
 {\it E-mail address: luismanuel.rivera@gmail.com}\\
\end{document}